\documentclass[leqno,12pt]{amsart}
\usepackage{amssymb}
\usepackage{amsmath}
\usepackage{enumerate}
\usepackage{amsfonts}
\usepackage{hyperref}
\usepackage{mathrsfs}

\usepackage[headheight=18pt, top=20mm, bottom=20mm, left=22mm, right=22mm]{geometry}

\usepackage{tikz}
\usepackage{pgfplots}
\pgfplotsset{compat=1.18}

\definecolor{darkgreen}{RGB}{45, 119, 75}

\newtheorem{theorem}{Theorem}[section]
\newtheorem{corollary}[theorem]{Corollary}
\newtheorem{lemma}[theorem]{Lemma}
\newtheorem{proposition}[theorem]{Proposition}
\newtheorem{remark}[theorem]{Remark}
\newtheorem{definition}[theorem]{Definition}

\numberwithin{equation}{section}

\hypersetup{
    colorlinks=true,
    linkcolor= blue,
    citecolor =cyan,
    urlcolor = teal,
}

\makeatletter
\@ifundefined{c@part}{\newcounter{part}}{}

\renewcommand\part{%
  \@ifstar{\@spart}{\@part}%
}

\def\@part#1{%
  \refstepcounter{part}%
  \addcontentsline{toc}{part}{\protect\numberline{\thepart}#1}%
  \par\vspace{1\bigskipamount}
  \begin{center}%
    \normalfont\large\bfseries Part\ \thepart\quad #1%
  \end{center}%
  \vspace{1\bigskipamount}%
  \@afterheading
}

\def\@spart#1{%
  \addcontentsline{toc}{part}{#1}%
  \par\vspace{1\bigskipamount}
  \begin{center}%
    \normalfont\large\bfseries #1%
  \end{center}%
  \vspace{1\bigskipamount}%
  \@afterheading
}
\makeatother
  
\begin{document}

\title[Remarks on Dunkl multipliers]
{A Remark on H\"ormander Multipliers \\ on Dunkl Hardy Spaces}


\author[Jacek Dziubański]{Jacek Dziubański}
\author[Agnieszka Hejna-Łyżwa]{Agnieszka Hejna-Łyżwa}
\begin{abstract}
Let $\mathcal R$ be a normalized root system in
$\mathbb R^N$ with a nonnegative multiplicity function
$k$, and let $\mathcal F$ be the associated Dunkl
transform. We prove a H\"ormander multiplier theorem on
the Hardy spaces $H^p_{\mathrm{Dunkl}}$, $0<p\le1$,
defined by conical Littlewood--Paley square functions. Let $W_2^s$ denote the classical Sobolev space in $\mathbb{R}^N$.
More precisely, if a multiplier $m$ satisfies
\[
    \sup_{t>0}
    \|\psi(\cdot)m(t\cdot)\|_{W_2^s}<\infty
\]
for a nonzero radial cutoff
$\psi\in C_c^\infty(\mathbb R^N\setminus\{0\})$ and
\[
    s>\mathbf N\left(\frac1p-\frac12\right),
\]
where $\mathbf N$ is the homogeneous dimension, then
the Dunkl multiplier
\[
    \mathcal T_mf=\mathcal F^{-1}(m\mathcal Ff)
\]
extends to a bounded operator on
$H^p_{\mathrm{Dunkl}}$.
\end{abstract}

\address{Jacek Dziubański, Uniwersytet Wroc\l awski,
Instytut Matematyczny,
Pl. Grunwaldzki 2,
50-384 Wroc\l aw,
Poland}
\email{jdziuban@math.uni.wroc.pl}

\address{Agnieszka Hejna-Łyżwa, Uniwersytet Wroc\l awski,
Instytut Matematyczny,
Pl. Grunwaldzki 2,
50-384 Wroc\l aw,
Poland}
\email{hejna@math.uni.wroc.pl}

\maketitle

\section{Introduction and statements of the results}\label{sec:intro}

On the Euclidean space $\mathbb R^N$, we consider a
normalized root system $R$ and a $G$-invariant multiplicity
function $k\ge0$, where $G$ is the finite reflection group
generated by the reflections $\sigma_\alpha$, $\alpha\in R$.
Let
\[
    dw(\mathbf x)
    =\prod_{\alpha\in R}
        |\langle\mathbf x,\alpha\rangle|^{k(\alpha)}
        \,d\mathbf x
\]
be the associated measure, where $d\mathbf x$ denotes
Lebesgue measure. We write
\[
    \mathbf N=N+\sum_{\alpha\in R}k(\alpha)
\]
for the homogeneous dimension.

Let $E(\mathbf x,\mathbf y)$ be the Dunkl kernel associated
with $(R,k)$. It extends uniquely to a holomorphic function
on $\mathbb C^N\times\mathbb C^N$. The Dunkl transform is
defined by
\begin{equation}\label{eq:DunklTransform}
    \mathcal Ff(\xi)
    =c_k^{-1}\int_{\mathbb R^N}
        E(-i\xi,\mathbf x)f(\mathbf x)\,dw(\mathbf x),
\end{equation}
where
\[
    c_k=\int_{\mathbb R^N}
        e^{-\|\mathbf x\|^2/2}\,dw(\mathbf x).
\]
Originally defined on $L^1(dw)$, it extends to a unitary
operator on $L^2(dw)$ and preserves the Schwartz space
$\mathcal S(\mathbb R^N)$; see \cite{deJeu}. Its inverse
is given by
\[
    \mathcal F^{-1}g(\mathbf x)
    =c_k^{-1}\int_{\mathbb R^N}
        E(i\xi,\mathbf x)g(\xi)\,dw(\xi).
\]

For $s\ge0$, we use the classical Sobolev norm
\[
    \|m\|_{W_2^s}
    =\bigl\|\widehat m(\mathbf x)
        (1+\|\mathbf x\|)^s\bigr\|_{L^2(d\mathbf x)},
\]
where
\[
    \widehat m(\mathbf x)
    =\int_{\mathbb R^N}
        e^{-i\langle\mathbf x,\xi\rangle}m(\xi)\,d\xi
\]
is the classical Fourier transform.

The metric measure space
$(\mathbb R^N,\|\mathbf x-\mathbf y\|,dw)$ is doubling;
see \eqref{eq:doubling}. For $0<p\le1$, let
$H^p_{\mathrm{Dunkl}}$ denote the Hardy space associated
with the Dunkl heat semigroup, defined through the
conical square function in Section~\ref{sec:Hardy}.
We denote its defining subspace in $L^2(dw)$ by
$\mathbb H^p_{\mathrm{Dunkl}}$. The operator atomic
decomposition and molecular characterization of these
spaces established in \cite{DzHL_atom} will be
essential to our argument.

A recent advance was obtained by Chang, Li, Wen, and Wu
\cite{CLWW}. They proved a H\"ormander multiplier theorem
for arbitrary finite reflection groups at the threshold
$s>\mathbf N/2$, including strong $L^q(dw)$ estimates for
$1<q<\infty$ and weak type $(1,1)$. Their argument is
based on a uniform estimate for the Dunkl kernel which
remains valid when the spectral parameter approaches
intersections of reflecting hyperplanes.

The purpose of this paper is to prove a H\"ormander
multiplier theorem on $H^p_{\mathrm{Dunkl}}$, $0<p \leq 1$,  under the smoothness
condition
\[
    s>\mathbf N\left(\frac1p-\frac12\right).
\]
This improves the sufficient regularity
$s>\mathbf N/p$ obtained in \cite{DzHL_atom}.
In particular, for $p=1$ the new condition becomes
$s>\mathbf N/2$. No radiality or $G$-invariance of the
multiplier is required.

\begin{theorem}\label{teo:teo_main}
Let $0<p\le1$, and let
$\psi\in C_c^\infty(\mathbb R^N\setminus\{0\})$
be a nonzero radial function. Suppose that a measurable
function $m$ on $\mathbb R^N$ satisfies
\begin{equation}\label{eq:assumption}
    M=\sup_{t>0}
        \|\psi(\cdot)m(t\cdot)\|_{W_2^s}<\infty
\end{equation}
for some
\[
    s>\mathbf N\left(\frac1p-\frac12\right).
\]
Then the Dunkl multiplier
\[
    \mathcal T_mf=\mathcal F^{-1}(m\mathcal Ff),
\]
initially considered on $\mathbb H^p_{\mathrm{Dunkl}}$,
has a unique bounded extension to
$H^p_{\mathrm{Dunkl}}$. Moreover,
\[
    \|\mathcal T_mf\|_{H^p_{\mathrm{Dunkl}}}
    \le CM\|f\|_{H^p_{\mathrm{Dunkl}}},
\]
where $C$ is independent of $m$ and $f$.
\end{theorem}

The smoothness condition in Theorem~\ref{teo:teo_main}
is the natural Hardy-space analogue of the classical
$L^2$ Sobolev threshold. In the Euclidean setting, the
classical H\"ormander multiplier theorem requires
Sobolev regularity greater than one half of the
dimension; see \cite{Hormander1960}.

\begin{remark}\label{rem:any_eta}
\normalfont
The condition \eqref{eq:assumption} is independent of
the choice of the nonzero radial cutoff. More precisely,
if it holds for $\psi$, then for every radial
$\eta\in C_c^\infty(\mathbb R^N\setminus\{0\})$,
\begin{equation}\label{eq:any_function}
    \sup_{t>0}
        \|\eta(\cdot)m(t\cdot)\|_{W_2^s}
    \le C_{\psi,\eta,s}M.
\end{equation}
This follows from a finite partition on the support
of $\eta$, using dilates of a region on which $\psi$
does not vanish, and boundedness of multiplication
by smooth compactly supported functions on $W_2^s$.
In particular, since $s>N/2$, Sobolev embedding gives
$\|m\|_{L^\infty}\le CM$, so $\mathcal T_m$ is well
defined and bounded on $L^2(dw)$.
\end{remark}

For $p=1$, Theorem~\ref{teo:teo_main} gives boundedness
on the intrinsic Dunkl Hardy space at the regularity
threshold $s>\mathbf N/2$. Earlier H\"ormander multiplier
results for the Dunkl transform were obtained in
\cite{DzH-JFA}; see also \cite{DzHL_atom} for the
corresponding multiplier theorem on
$H^p_{\mathrm{Dunkl}}$ under a stronger smoothness
condition. For other recent
multiplier results under modified H\"ormander conditions,
including estimates for radial inputs, see Mukherjee and
Thangavelu \cite{Thangavelu2}.

The purpose of the present note is to record the
observation that the kernel estimates and the dyadic
argument of \cite{CLWW} also lead to boundedness on the
intrinsic Hardy spaces $H^p_{\mathrm{Dunkl}}$,
$0<p\le1$, when they are combined with the operator
atomic decomposition and the molecular criterion
established in \cite{DzHL_atom}. Thus the analytic input
of our proof is provided by \cite{CLWW}; our argument
consists in placing their estimates into the atomic and
molecular framework of \cite{DzHL_atom}. In this sense,
Theorem~\ref{teo:teo_main} may be viewed as a
Hardy-space consequence of the results of
\cite{CLWW}.

We now describe more precisely the estimates used in
this observation. The first one is the joint
Dunkl-kernel estimate obtained in
\cite[Proposition 1.2]{CLWW}. Set
\begin{equation}\label{eq:omega_Omega}
    w(\xi)
    =\prod_{\alpha\in R}
        |\langle\xi,\alpha\rangle|^{k(\alpha)},
    \qquad
    \Omega(\mathbf y)
    =\prod_{\alpha\in R}
        \bigl(1+|\langle\mathbf y,\alpha\rangle|\bigr)^{k(\alpha)}.
\end{equation}
For every fixed compact annulus
$\mathcal A\subset\mathbb R^N\setminus\{0\}$,
the estimate of \cite{CLWW} gives
\begin{equation}\label{eq:joint_kernel_input}
    w(\xi)|E(-i\xi,\mathbf y)|^2
    \le C_{\mathcal A}\Omega(\mathbf y)^{-1},
    \qquad
    \xi\in\mathcal A\setminus
        \bigcup_{\alpha\in R}\alpha^\perp,
    \quad \mathbf y\in\mathbb R^N.
\end{equation}
Moreover,
\begin{equation}\label{eq:Omega_ball_comparison}
    \Omega(\mathbf y)\asymp w(B(\mathbf y,1)),
    \qquad \mathbf y\in\mathbb R^N.
\end{equation}
Thus \eqref{eq:joint_kernel_input} provides the
appropriate volume normalization for $L^2$ estimates of
localized multiplier kernels. The estimate remains
uniform as the spectral variable approaches the
reflecting hyperplanes.

Proposition~4.3 of \cite{CLWW} combines
\eqref{eq:joint_kernel_input} with the relevant
integration-by-parts identities and gives weighted
$L^2$ estimates for multiplier kernels whose symbols
are supported in a fixed annulus. We use this proposition
as stated in \cite{CLWW}. The only additional step needed
here is to write its estimate at the dyadic scale relevant
to the multiplier decomposition.

More precisely, let $\phi$ be a smooth radial annular
cutoff forming a dyadic partition of unity, and set
\[
    m_\ell(\xi)=m(\xi)\phi(2^{-\ell}\xi).
\]
If $\mathscr K_\ell$ denotes the integral kernel of
$\mathcal T_{m_\ell}$, then, for $0\le\sigma\le s$,
\[
    \bigl\|(1+2^\ell d(\cdot,\mathbf y))^\sigma
        \mathscr K_\ell(\cdot,\mathbf y)\bigr\|_{L^2(dw)}
    \le CM\,w(B(\mathbf y,2^{-\ell}))^{-1/2}.
\]
The constant is independent of $\ell$ and $\mathbf y$.
This is the rescaled form of the annular estimate from
\cite[Proposition 4.3]{CLWW}; the scaling is recorded in
Lemma~\ref{lem:weighted_annular_kernel} and
Corollary~\ref{cor:weighted_dyadic_kernel}.

It remains to observe that these weighted kernel
estimates fit the operator molecular framework developed
in \cite{DzHL_atom}. For completeness, recall that the generalized Dunkl
translations are defined, initially for Schwartz
functions, by
\begin{equation}\label{eq:translation}
    \tau_{\mathbf x}f(\mathbf y)
    =c_k^{-1}\int_{\mathbb R^N}
        E(i\xi,\mathbf x)E(i\xi,\mathbf y)
        \mathcal Ff(\xi)\,dw(\xi).
\end{equation}
They extend to contractions on $L^2(dw)$.

The paper is organized as follows. In
Section~\ref{sec:preliminaries}, we recall the basic
facts from Dunkl theory needed in the sequel. In
Section~\ref{sec:Hardy}, we present the operator atomic
decomposition and molecular characterization of the
spaces $H^p_{\mathrm{Dunkl}}$. Section~\ref{sec:weighted}
is devoted to weighted $L^2$ estimates for dyadically
localized multiplier kernels. Finally, in
Section~\ref{sec:Hormander}, we combine these estimates
with the atomic and molecular theory to prove the main
H\"ormander multiplier theorem.

\section{Preliminaries: Dunkl theory}\label{sec:preliminaries}

We retain the notation introduced in Section~\ref{sec:intro}.
For background on Dunkl operators and the Dunkl transform,
we refer to \cite{Dunkl,Roesler3,Roesler-Voit,deJeu}.

\subsection{The measure and the orbit distance}

The roots are normalized by $\|\alpha\|^2=2$, and
\[
    \sigma_\alpha(\mathbf x)
    =\mathbf x
       -2\frac{\langle\mathbf x,\alpha\rangle}
                {\|\alpha\|^2}\alpha,
    \qquad \alpha\in R.
\]
The measure $dw$ is $G$-invariant and homogeneous:
\begin{equation}\label{eq:t_ball}
    w(B(t\mathbf x,tr))
    =t^{\mathbf N}w(B(\mathbf x,r)),
    \qquad t,r>0.
\end{equation}
The measures of balls satisfy
\begin{equation}\label{eq:balls_asymp}
    w(B(\mathbf x,r))
    \asymp
    r^N\prod_{\alpha\in R}
       \bigl(|\langle\mathbf x,\alpha\rangle|+r\bigr)^{k(\alpha)}.
\end{equation}
In particular,
\begin{equation}\label{eq:growth}
    w(B(\mathbf x,\rho))
    \le C\left(\frac{\rho}{r}\right)^{\mathbf N}
       w(B(\mathbf x,r)),
    \qquad 0<r\le\rho,
\end{equation}
and
\begin{equation}\label{eq:doubling}
    w(B(\mathbf x,2r))\le Cw(B(\mathbf x,r)).
\end{equation}

For a set $A\subseteq\mathbb R^N$, write
\[
    \mathcal O(A)=\bigcup_{g\in G}gA.
\]
The orbit distance is
\begin{equation}\label{eq:distance_of_orbits}
    d(\mathbf x,\mathbf y)
    =\min_{g\in G}\|\mathbf x-g(\mathbf y)\|.
\end{equation}
It is symmetric, satisfies the triangle inequality,
and is $G$-invariant in each variable. Moreover,
\[
    d(t\mathbf x,t\mathbf y)=t\,d(\mathbf x,\mathbf y),
    \qquad t>0.
\]
For open balls,
\[
    \{\mathbf y:d(\mathbf x,\mathbf y)<r\}
    =\mathcal O(B(\mathbf x,r)),
\]
and
\begin{equation}\label{eq:orbit_ball_measure}
    w(B(\mathbf x,r))
    \le w(\mathcal O(B(\mathbf x,r)))
    \le |G|w(B(\mathbf x,r)).
\end{equation}
The same statements hold for closed balls with the
corresponding non-strict inequalities; this distinction
does not affect the integral estimates below.

We shall also use the consequence of doubling and
$G$-invariance that
\begin{equation}\label{eq:nearby_ball_comparison}
    w(B(\mathbf x,r))\asymp w(B(\mathbf y,r))
    \qquad\text{if }d(\mathbf x,\mathbf y)\le r.
\end{equation}

\subsection{Dunkl operators and transform identities}

For $j=1,\ldots,N$, the coordinate Dunkl operators are
\begin{equation}\label{eq:T_xi}
    T_jf(\mathbf x)
    =\partial_jf(\mathbf x)
       +\frac12\sum_{\alpha\in R}
          k(\alpha)\alpha_j
          \frac{f(\mathbf x)-f(\sigma_\alpha\mathbf x)}
               {\langle\alpha,\mathbf x\rangle}.
\end{equation}
They commute and preserve $\mathcal S(\mathbb R^N)$.
Their integration-by-parts formula is
\begin{equation}\label{eq:Dunkl_integration_by_parts}
    \int_{\mathbb R^N}(T_jf)(\mathbf x)g(\mathbf x)
       \,dw(\mathbf x)
    =
    -\int_{\mathbb R^N}f(\mathbf x)(T_jg)(\mathbf x)
       \,dw(\mathbf x),
\end{equation}
for $f,g\in\mathcal S(\mathbb R^N)$, and also for
smooth functions when one factor is compactly
supported. For $f,g \in C^1(\mathbb{R}^N)$, we have the following Leibniz-type rule
\begin{equation}\label{eq:general_Leibniz}
    T_j(fg)(\mathbf{x})=(T_jf)(\mathbf{x})g(\mathbf{x})+f(\mathbf{x})\partial_{j}g(\mathbf{x})+\sum_{\alpha \in R}\frac{k(\alpha)}{2} \langle\alpha,e_j\rangle f(\sigma_{\alpha}(\mathbf{x}))\frac{g(\mathbf{x})-g(\sigma_{\alpha}(\mathbf{x}))}{\langle \mathbf{x}, \alpha \rangle}.
\end{equation}

The Dunkl kernel satisfies
\begin{equation}\label{eq:eigenfunction_identity}
    T_j^{\mathbf x}E(\mathbf x,\mathbf z)
    =z_jE(\mathbf x,\mathbf z),
\end{equation}
as well as:
\begin{equation}\label{eq:Dunkl_kernel_identities}
\begin{split}
    E(\mathbf z,\mathbf v)&=E(\mathbf v,\mathbf z),\\
    E(g\mathbf z,g\mathbf v)&=E(\mathbf z,\mathbf v),
       \qquad g\in G,\\
    E(t\mathbf z,\mathbf v)&=E(\mathbf z,t\mathbf v),
       \qquad t>0.
\end{split}
\end{equation}
For real $\xi,\mathbf x$,
\begin{equation}\label{eq:E}
    |E(i\xi,\mathbf x)|\le1.
\end{equation}

With the normalization in \eqref{eq:DunklTransform},
Plancherel's theorem reads
\begin{equation}\label{eq:Plancherel}
    \|\mathcal Ff\|_{L^2(dw)}=\|f\|_{L^2(dw)}.
\end{equation}
For Schwartz functions,
\begin{equation}\label{eq:T_j_transform}
    \mathcal F(T_jf)(\xi)=i\xi_j\mathcal Ff(\xi),
\end{equation}
and
\begin{equation}\label{eq:inverse_coordinate_identity}
    x_j\mathcal F^{-1}a(\mathbf x)
    =\mathcal F^{-1}(iT_ja)(\mathbf x).
\end{equation}
These identities and
\eqref{eq:Dunkl_integration_by_parts} will be used
to obtain weighted estimates for the multiplier
kernels.

The Dunkl Laplacian is
\begin{equation}\label{eq:Dunkl_laplacian_form}
    \Delta_k=\sum_{j=1}^N T_j^2.
\end{equation}
Consequently,
\begin{equation}\label{eq:Laplacian_on_Fourier_side}
    \mathcal F(\Delta_kf)(\xi)
    =-\|\xi\|^2\mathcal Ff(\xi).
\end{equation}
We use its self-adjoint realization on $L^2(dw)$
given by this multiplier identity. The domains of
its powers are recalled in \eqref{eq:domain}.

\subsection{The heat semigroup and square-function kernels}
\label{sec:Qt}

The Dunkl heat semigroup $H_t=e^{t\Delta_k}$ satisfies
\begin{equation}\label{eq:heat_transform}
    \mathcal F(H_tf)(\xi)
    =e^{-t\|\xi\|^2}\mathcal Ff(\xi),
    \qquad f\in L^2(dw).
\end{equation}
Let $Q_t(\mathbf x,\mathbf y)$ denote the integral
kernel of $t^2\Delta_kH_{t^2}$, so that
\[
    t^2\Delta_kH_{t^2}f(\mathbf x)
    =\int_{\mathbb R^N}
       Q_t(\mathbf x,\mathbf y)f(\mathbf y)\,dw(\mathbf y).
\]
The kernel is symmetric:
\[
    Q_t(\mathbf x,\mathbf y)=Q_t(\mathbf y,\mathbf x).
\]
The heat-kernel derivative estimates imply
\begin{equation}\label{eq:Qt-bound}
    |Q_t(\mathbf x,\mathbf y)|
    \le
    \frac{C}{w(B(\mathbf x,t))}
       \exp\left(-c\frac{d(\mathbf x,\mathbf y)^2}{t^2}\right),
    \qquad t>0;
\end{equation}
see \cite[equation (2.30)]{DzHL_atom}.
This is the only pointwise heat-kernel estimate
needed in the molecular argument.

We also record the uniform spectral bound
\begin{equation}\label{eq:Q_operator_bounded}
    \bigl\|t^{2J}\Delta_k^J
       (t^2\Delta_kH_{t^2})\bigr\|_{L^2(dw)\to L^2(dw)}
    \le C_J,
    \qquad t>0,\quad J\in\mathbb N_0.
\end{equation}
Indeed, by \eqref{eq:heat_transform} and
\eqref{eq:Laplacian_on_Fourier_side}, its multiplier is
\[
    (-1)^{J+1}
       (t\|\xi\|)^{2J+2}e^{-t^2\|\xi\|^2}.
\]
Thus \eqref{eq:Q_operator_bounded} follows from
Plancherel's theorem and boundedness of
$u^{2J+2}e^{-u^2}$ on $[0,\infty)$.

\section{Preliminaries: Hardy spaces}\label{sec:Hardy}

In this section we recall the operator atomic and
molecular results from \cite{DzHL_atom} that will
be used in the proof of Theorem~\ref{teo:teo_main}. The integer describing the order of an atom
or molecule will be denoted by $J$, to distinguish it
from the constant $M$ in \eqref{eq:assumption}.

\subsection{The square function and the Hardy spaces}

Let $H_t=e^{t\Delta_k}$ be the Dunkl heat semigroup.
For $f\in L^2(dw)$, define
\begin{equation}\label{eq:square_conic}
    Sf(\mathbf x)
    =
    \left(
       \int_0^\infty
       \int_{\|\mathbf x-\mathbf y\|<t}
       |t^2\Delta_kH_{t^2}f(\mathbf y)|^2
       \frac{dw(\mathbf y)}{w(B(\mathbf x,t))}
       \frac{dt}{t}
    \right)^{1/2}.
\end{equation}
Following \cite[Definition 1.1 and
equations (1.1)--(1.2)]{DzHL_atom}, for $0<p\le1$
we set
\[
    \mathbb H^p_{\mathrm{Dunkl}}
    =
    \{f\in L^2(dw):Sf\in L^p(dw)\},
    \qquad
    \|f\|_{\mathbb H^p_{\mathrm{Dunkl}}}
    =\|Sf\|_{L^p(dw)}.
\]
The space $H^p_{\mathrm{Dunkl}}$ is the completion of
$\mathbb H^p_{\mathrm{Dunkl}}$ in this quasi-norm.

We record two elementary properties of $S$ that are
needed when passing from estimates on atoms to
estimates on the whole space:
\begin{equation}\label{eq:Hardy_square_difference}
    |Sf-Sg|\le S(f-g)
\end{equation}
and
\begin{equation}\label{eq:Hardy_square_L2}
    \|Sf\|_{L^2(dw)}\asymp\|f\|_{L^2(dw)}.
\end{equation}
The first follows from the reverse triangle inequality
in the Hilbert space defining the cone integral.
For the second, Fubini's theorem and doubling give
\[
    \|Sf\|_{L^2(dw)}^2
    \asymp
    \int_0^\infty
       \|t^2\Delta_kH_{t^2}f\|_{L^2(dw)}^2\,\frac{dt}{t}.
\]
Indeed, for fixed $\mathbf y,t$,
\[
    \int_{\|\mathbf x-\mathbf y\|<t}
       \frac{dw(\mathbf x)}{w(B(\mathbf x,t))}
    \asymp1.
\]
Plancherel's theorem and
\[
    \mathcal F(t^2\Delta_kH_{t^2}f)(\xi)
    =-t^2\|\xi\|^2e^{-t^2\|\xi\|^2}\mathcal Ff(\xi)
\]
then prove \eqref{eq:Hardy_square_L2}, since
\[
    \int_0^\infty
       t^4\|\xi\|^4e^{-2t^2\|\xi\|^2}\,\frac{dt}{t}
    =\frac18,
    \qquad \xi\ne0.
\]

In particular, for $0<p\le1$,
\begin{equation}\label{eq:Hardy_p_subadditivity}
    \|f+g\|_{\mathbb H^p_{\mathrm{Dunkl}}}^p
    \le
    \|f\|_{\mathbb H^p_{\mathrm{Dunkl}}}^p
    +\|g\|_{\mathbb H^p_{\mathrm{Dunkl}}}^p.
\end{equation}
Thus $\|f-g\|_{\mathbb H^p_{\mathrm{Dunkl}}}^p$
defines a metric.

The completion can also be identified with a space
of tempered distributions, using the natural pairing
\[
    \langle f,\varphi\rangle
    =\int_{\mathbb R^N}f(\mathbf x)\varphi(\mathbf x)
       \,dw(\mathbf x).
\]
The existence of this identification and its
injectivity follow from
\cite[Corollaries 4.12--4.13 and
Proposition 4.14]{DzHL_atom}.
For the multiplier argument below, however, it is
enough to work first on
$\mathbb H^p_{\mathrm{Dunkl}}$ and then pass to its
completion.

\subsection{Operator atoms and atomic decompositions}

For a positive integer $J$, recall that
\begin{equation}\label{eq:domain}
    \mathcal D(\Delta_k^J)
    =
    \left\{
       f\in L^2(dw):
       \int_{\mathbb R^N}
          \|\xi\|^{4J}|\mathcal Ff(\xi)|^2\,dw(\xi)
       <\infty
    \right\};
\end{equation}
see \cite[equation (2.28)]{DzHL_atom}.
On this domain,
\[
    \mathcal F(\Delta_k^Jf)(\xi)
    =(-1)^J\|\xi\|^{2J}\mathcal Ff(\xi).
\]

\begin{definition}[\cite{DzHL_atom}, Definition 4.1]
\label{def:operator_atom}
Let $0<p\le1$ and let $J$ be an integer satisfying
\[
    J>\frac{\mathbf N(2-p)}{4p}.
\]
A function $\boldsymbol a$ is a
$(p,2,J,\Delta_k)$-atom associated with a ball
$B=B(\mathbf x_0,r)$ if there exists
$\boldsymbol b\in\mathcal D(\Delta_k^J)$ such that
\begin{equation}\label{eq:Hardy_atom_representation}
    \boldsymbol a=\Delta_k^J\boldsymbol b,
\end{equation}
\begin{equation}\label{eq:Hardy_atom_support}
    \operatorname{supp}\boldsymbol b\subseteq\mathcal O(B),
\end{equation}
and
\begin{equation}\label{eq:Hardy_atom_size}
    \|(r^2\Delta_k)^n\boldsymbol b\|_{L^2(dw)}
    \le r^{2J}w(B)^{1/2-1/p},
    \qquad n=0,\ldots,J.
\end{equation}
\end{definition}

Since $\mathcal O(B)$ is $G$-invariant, Dunkl
differentiation preserves distributional support
in this set. Consequently,
\[
    \operatorname{supp}(\Delta_k^n\boldsymbol b)
    \subseteq\mathcal O(B),
    \qquad n=0,\ldots,J.
\]
In particular,
\begin{equation}\label{eq:Hardy_atom_L2_size}
    \operatorname{supp}\boldsymbol a\subseteq\mathcal O(B),
    \qquad
    \|\boldsymbol a\|_{L^2(dw)}\le w(B)^{1/2-1/p}.
\end{equation}

The precise atomic decomposition needed below is
the following $L^2$ version.

\begin{proposition}[\cite{DzHL_atom}, Proposition 4.11]
\label{teo:operator_atomic_decomposition}
Let $0<p\le1$ and let
$J>\mathbf N(2-p)/(4p)$ be an integer.
For every $f\in\mathbb H^p_{\mathrm{Dunkl}}$, there
exist $(p,2,J,\Delta_k)$-atoms $\boldsymbol a_\nu$
and coefficients $c_\nu\in\mathbb C$ such that
\begin{equation}\label{eq:Hardy_atomic_decomposition}
    f=\sum_{\nu=1}^\infty c_\nu\boldsymbol a_\nu
    \quad\text{in }L^2(dw),
\end{equation}
and
\begin{equation}\label{eq:Hardy_atomic_coefficients}
    \sum_{\nu=1}^\infty|c_\nu|^p
    \le C\|f\|_{\mathbb H^p_{\mathrm{Dunkl}}}^p.
\end{equation}
The series also converges in $\mathcal S'(\mathbb R^N)$.
The constant depends only on the fixed parameters
and is independent of $f$.
\end{proposition}

\subsection{Operator molecules}

For a ball $B=B(\mathbf x_0,r)$, set
\begin{equation}\label{eq:Hardy_orbit_annuli}
    U_0(B)=\mathcal O(B),
    \qquad
    U_j(B)=\mathcal O(2^jB)\setminus
              \mathcal O(2^{j-1}B),\quad j\ge1.
\end{equation}
These sets form a disjoint decomposition of
$\mathbb R^N$, up to boundaries of measure zero.
Moreover,
\[
    w(U_j(B))\le |G|w(2^jB).
\]
For $j\ge2$, $\mathbf x\in U_j(B)$ and
$\mathbf y\in\mathcal O(B)$, we have
\begin{equation}\label{eq:Hardy_annular_separation}
    d(\mathbf x,\mathbf y)
    \ge (2^{j-1}-1)r\ge2^{j-2}r.
\end{equation}

We use the molecular conditions of
\cite[Definition 4.3]{DzHL_atom} with the
disjoint orbit annuli \eqref{eq:Hardy_orbit_annuli}.

Molecular decompositions of classical Hardy spaces were
introduced by Taibleson and Weiss
\cite{Taibleson-Weiss}. Operator-adapted Hardy spaces and
their atomic and molecular descriptions were developed
in a general functional-calculus setting in
\cite{Hofman,Duong}. In the rational Dunkl setting, the
particular operator atoms and molecules used here were
introduced and studied in \cite{DzHL_atom}.

\begin{definition}\label{def:molecule}
Let $0<p\le1$, $\varepsilon>0$, and let
$J>\mathbf N(2-p)/(4p)$ be an integer.
An $L^2(dw)$ function $\boldsymbol\mu$ is a
$(p,2,J,\Delta_k,\varepsilon)$-molecule associated
with $B=B(\mathbf x_0,r)$ if there exists
$\boldsymbol b\in\mathcal D(\Delta_k^J)$ such that
\[
    \boldsymbol\mu=\Delta_k^J\boldsymbol b
\]
and
\begin{equation}\label{eq:molecule_condition}
    \|(r^2\Delta_k)^n\boldsymbol b\|_{L^2(U_j(B))}
    \le r^{2J}2^{-j\varepsilon}
       w(2^jB)^{1/2-1/p},
    \qquad
    j\ge0,\quad n=0,\ldots,J.
\end{equation}
\end{definition}

The molecular condition immediately implies the
following tail estimate, as in
\cite[Lemma 4.4]{DzHL_atom}:
\begin{equation}\label{eq:Hardy_molecular_tail}
    \|(r^2\Delta_k)^n\boldsymbol b\|_
       {L^2(\bigcup_{j\ge j_0}U_j(B))}
    \le Cr^{2J}2^{-j_0\varepsilon}
       w(2^{j_0}B)^{1/2-1/p},
\end{equation}
for $j_0\ge0$ and $0\le n\le J$.
Indeed, square \eqref{eq:molecule_condition}, sum
over $j\ge j_0$, and use $1-2/p<0$ together with
$w(2^jB)\ge w(2^{j_0}B)$.
In particular,
\begin{equation}\label{eq:Hardy_molecular_global_L2}
    \|\boldsymbol\mu\|_{L^2(dw)}\le Cw(B)^{1/2-1/p},
    \qquad
    \|\boldsymbol b\|_{L^2(dw)}\le Cr^{2J}w(B)^{1/2-1/p}.
\end{equation}

The following is the molecular estimate from
\cite[Proposition 4.5]{DzHL_atom}, formulated
with \eqref{eq:Hardy_orbit_annuli}. We include the
argument to specify the annular convention.

\begin{proposition}\label{prop:m_in_Hp}
Let $0<p\le1$, $\varepsilon>0$, and let
$J>\mathbf N(2-p)/(4p)$ be an integer.
Every $(p,2,J,\Delta_k,\varepsilon)$-molecule
$\boldsymbol\mu$ belongs to
$\mathbb H^p_{\mathrm{Dunkl}}$, and
\[
    \|\boldsymbol\mu\|_{\mathbb H^p_{\mathrm{Dunkl}}}
    \le C.
\]
The constant may depend on $R,k,p,J,\varepsilon$,
but is independent of the molecule and its
associated ball.
\end{proposition}

\begin{proof}
Write $\boldsymbol\mu=\Delta_k^J\boldsymbol b$,
and let $B=B(\mathbf x_0,r)$ be its associated ball.
Set
\[
    A_p=\mathbf N\left(\frac1p-\frac12\right).
\]
Since $2J>A_p$, we may choose $\theta\in(0,1)$
such that
\[
    \delta_0:=2J\theta-A_p>0.
\]
We shall prove an annular square-function estimate
with exponent
\[
    \delta=\min\{\varepsilon,\delta_0\}>0.
\]

For brevity, write
$
    \mathcal Q_t=t^2\Delta_kH_{t^2}.
$. For $T>0$, define
\[
    S_{<T}f(\mathbf x)
    =
    \left(
       \int_0^T
       \int_{\|\mathbf x-\mathbf y\|<t}
       |\mathcal Q_tf(\mathbf y)|^2
       \frac{dw(\mathbf y)}{w(B(\mathbf x,t))}
       \frac{dt}{t}
    \right)^{1/2},
\]
and define $S_{\ge T}$ by replacing $(0,T)$ with
$[T,\infty)$. Thus
\[
    (Sf)^2=(S_{<T}f)^2+(S_{\ge T}f)^2.
\]

We first record the consequence of doubling that
\begin{equation}\label{eq:molecule_cone_integration}
    \int_{\|\mathbf x-\mathbf y\|<t}
       \frac{dw(\mathbf x)}{w(B(\mathbf x,t))}
    \le C,
    \qquad \mathbf y\in\mathbb R^N,\quad t>0.
\end{equation}
Indeed, whenever $\|\mathbf x-\mathbf y\|<t$,
the measures $w(B(\mathbf x,t))$ and
$w(B(\mathbf y,t))$ are comparable.

The contribution from $\mathcal O(16B)$ follows
from H\"older's inequality,
\eqref{eq:Hardy_square_L2}, and
\eqref{eq:Hardy_molecular_global_L2}:
\begin{align*}
    \int_{\mathcal O(16B)}|S\boldsymbol\mu|^p\,dw
    &\le
       w(\mathcal O(16B))^{1-p/2}
       \|S\boldsymbol\mu\|_{L^2(dw)}^p\le
       Cw(B)^{1-p/2}\|\boldsymbol\mu\|_{L^2(dw)}^p
       \le C.
\end{align*}
Here we used $w(\mathcal O(16B))\le Cw(B)$.

Fix $j\ge5$ and set
\[
    T_j=2^{\theta(j-5)}r.
\]
We estimate the large-time and small-time parts
separately.

\medskip
\noindent
\textit{Large times.}
Since $\boldsymbol\mu=\Delta_k^J\boldsymbol b$,
the spectral estimate
\eqref{eq:Q_operator_bounded}
(see also \cite[Lemma 2.5]{DzHL_atom}) gives
\[
    \|\mathcal Q_t\boldsymbol\mu\|_{L^2(dw)}
    =\|\Delta_k^J\mathcal Q_t\boldsymbol b\|_{L^2(dw)}
    \le Ct^{-2J}\|\boldsymbol b\|_{L^2(dw)}.
\]
Fubini's theorem and
\eqref{eq:molecule_cone_integration} therefore yield
\begin{align*}
    &\|S_{\ge T_j}\boldsymbol\mu\|_{L^2(U_j(B))}^2\le C\int_{T_j}^\infty
       \|\mathcal Q_t\boldsymbol\mu\|_{L^2(dw)}^2\,\frac{dt}{t}\le C\|\boldsymbol b\|_{L^2(dw)}^2
       \int_{T_j}^\infty t^{-4J}\,\frac{dt}{t}\\
    &\le CT_j^{-4J}r^{4J}w(B)^{1-2/p}=C2^{-4J\theta(j-5)}w(B)^{1-2/p}.
\end{align*}
Since $1-2/p<0$, the growth estimate implies
\begin{equation}\label{eq:molecule_volume_conversion}
    w(B)^{1-2/p}
    =
    \left(\frac{w(2^jB)}{w(B)}\right)^{2/p-1}
       w(2^jB)^{1-2/p}
    \le C2^{2jA_p}w(2^jB)^{1-2/p}.
\end{equation}
Consequently,
\begin{equation}\label{eq:molecule_large_times}
    \|S_{\ge T_j}\boldsymbol\mu\|_{L^2(U_j(B))}
    \le C2^{-j\delta_0}w(2^jB)^{1/2-1/p}.
\end{equation}

\medskip
\noindent
\textit{Small times: the far part of the molecule.}
Decompose
\[
    \boldsymbol\mu
    =\boldsymbol\mu_{\mathrm{far}}
       +\boldsymbol\mu_{\mathrm{near}},
\]
where
\[
    \boldsymbol\mu_{\mathrm{far}}
    =\boldsymbol\mu\,
       \chi_{\mathbb R^N\setminus\mathcal O(2^{j-3}B)},
    \qquad
    \boldsymbol\mu_{\mathrm{near}}
    =\boldsymbol\mu\,\chi_{\mathcal O(2^{j-3}B)}.
\]
The molecular tail estimate
\eqref{eq:Hardy_molecular_tail}, applied with
$n=J$ and $j_0=j-2$, gives
\[
    \|\boldsymbol\mu_{\mathrm{far}}\|_{L^2(dw)}
    \le
       C2^{-(j-2)\varepsilon}
       w(2^{j-2}B)^{1/2-1/p}
    \le C2^{-j\varepsilon}w(2^jB)^{1/2-1/p}.
\]
The last step uses doubling over a fixed number
of scales. By \eqref{eq:Hardy_square_L2},
\begin{equation}\label{eq:molecule_small_times_far}
    \|S_{<T_j}\boldsymbol\mu_{\mathrm{far}}\|_{L^2(U_j(B))}
    \le C\|\boldsymbol\mu_{\mathrm{far}}\|_{L^2(dw)}
    \le C2^{-j\varepsilon}w(2^jB)^{1/2-1/p}.
\end{equation}

\medskip
\noindent
\textit{Small times: the near part of the molecule.}
We first explain the off-diagonal estimate used here.
The Gaussian bound \eqref{eq:Qt-bound}, together
with doubling and the orbit-ball measure estimate,
implies
\[
    \int_{\mathbb R^N}
       \exp\left(-c\frac{d(\mathbf y,\mathbf z)^2}{t^2}\right)
       \,dw(\mathbf z)
    \le Cw(B(\mathbf y,t)).
\]
This follows by splitting the integral into
$d(\mathbf y,\mathbf z)<t$ and the dyadic orbit
annuli of radii $2^\nu t$.

Let $E,F$ be measurable sets such that
$d(\mathbf y,\mathbf z)\ge\rho$ for
$\mathbf y\in E$ and $\mathbf z\in F$.
Splitting the Gaussian exponential into two factors
gives
\[
    \sup_{\mathbf y\in E}
       \int_F|Q_t(\mathbf y,\mathbf z)|\,dw(\mathbf z)
    \le C\exp\left(-c\frac{\rho^2}{t^2}\right).
\]
By symmetry of $Q_t$, the same bound holds with
$E$ and $F$ interchanged. Schur's test therefore gives
\begin{equation}\label{eq:molecule_off_diagonal}
    \|\chi_E\mathcal Q_t\chi_F\|_{L^2(dw)\to L^2(dw)}
    \le C\exp\left(-c\frac{\rho^2}{t^2}\right).
\end{equation}
The underlying Gaussian estimate is also recorded
in \cite[equation (2.30)]{DzHL_atom}.

For $0<t<T_j$, set
\[
    E_{j,t}
    =
    \{\mathbf y\in\mathbb R^N:
       \|\mathbf x-\mathbf y\|<t
       \text{ for some }\mathbf x\in U_j(B)\}.
\]
Since $j\ge5$ and $\theta<1$,
\[
    T_j\le2^{j-5}r.
\]
If $\mathbf y\in E_{j,t}$ and
$\mathbf z\in\mathcal O(2^{j-3}B)$, choose
$\mathbf x\in U_j(B)$ with
$\|\mathbf x-\mathbf y\|<t$. The triangle inequality
for the orbit distance gives
\begin{align*}
    d(\mathbf y,\mathbf z)
    &\ge
       d(\mathbf x,\mathbf x_0)
       -d(\mathbf x,\mathbf y)
       -d(\mathbf z,\mathbf x_0)\ge
       2^{j-1}r-t-2^{j-3}r
       \ge2^{j-2}r.
\end{align*}
Thus \eqref{eq:molecule_off_diagonal}, with
$E=E_{j,t}$ and $F=\mathcal O(2^{j-3}B)$, implies
\[
    \|\chi_{E_{j,t}}\mathcal Q_t
          \boldsymbol\mu_{\mathrm{near}}\|_{L^2(dw)}
    \le
       C\exp\left(-c\frac{(2^jr)^2}{t^2}\right)
       \|\boldsymbol\mu_{\mathrm{near}}\|_{L^2(dw)}.
\]
By Fubini's theorem and
\eqref{eq:molecule_cone_integration},
\begin{align*}
    \|S_{<T_j}\boldsymbol\mu_{\mathrm{near}}\|_{L^2(U_j(B))}^2
    &\le
       C\int_0^{T_j}
       \|\chi_{E_{j,t}}\mathcal Q_t
          \boldsymbol\mu_{\mathrm{near}}\|_{L^2(dw)}^2
       \,\frac{dt}{t}\le
       C\|\boldsymbol\mu\|_{L^2(dw)}^2
       \int_0^{T_j}
          \exp\left(-c\frac{(2^jr)^2}{t^2}\right)
          \frac{dt}{t}.
\end{align*}

Choose $K>0$ such that
\[
    K(1-\theta)>A_p+\delta.
\]
Using $e^{-cv^2}\le C_Kv^{-2K}$ for $v>0$, we obtain
\begin{align*}
    \int_0^{T_j}
       \exp\left(-c\frac{(2^jr)^2}{t^2}\right)\frac{dt}{t}
    &\le C_K\int_0^{T_j}
       \left(\frac{t}{2^jr}\right)^{2K}\frac{dt}{t}\le C_K\left(\frac{T_j}{2^jr}\right)^{2K}\le C_K2^{-2K(1-\theta)j}.
\end{align*}
Combining this with
\eqref{eq:Hardy_molecular_global_L2} and
\eqref{eq:molecule_volume_conversion}, we conclude that
\begin{equation}\label{eq:molecule_small_times_near}
\begin{split}
    \|S_{<T_j}\boldsymbol\mu_{\mathrm{near}}\|_{L^2(U_j(B))}
    &\le
       C2^{-j(K(1-\theta)-A_p)}
       w(2^jB)^{1/2-1/p}\le C2^{-j\delta}w(2^jB)^{1/2-1/p}.
\end{split}
\end{equation}

Finally, the triangle inequality in the cone integrals
gives
\[
    S\boldsymbol\mu
    \le S_{\ge T_j}\boldsymbol\mu
       +S_{<T_j}\boldsymbol\mu_{\mathrm{far}}
       +S_{<T_j}\boldsymbol\mu_{\mathrm{near}}.
\]
Hence \eqref{eq:molecule_large_times},
\eqref{eq:molecule_small_times_far}, and
\eqref{eq:molecule_small_times_near} yield
\[
    \|S\boldsymbol\mu\|_{L^2(U_j(B))}
    \le C2^{-j\delta}w(2^jB)^{1/2-1/p},
    \qquad j\ge5.
\]
Since $w(U_j(B))\le |G|w(2^jB)$, H\"older's inequality
implies
\begin{align*}
    \sum_{j\ge5}\int_{U_j(B)}|S\boldsymbol\mu|^p\,dw
    &\le
       \sum_{j\ge5}
       \|S\boldsymbol\mu\|_{L^2(U_j(B))}^p
       w(U_j(B))^{1-p/2}\le C\sum_{j\ge5}2^{-j\delta p}<\infty.
\end{align*}
Together with the estimate on $\mathcal O(16B)$,
this proves $\|S\boldsymbol\mu\|_p\le C$.
As $\boldsymbol\mu\in L^2(dw)$ by definition,
the proposition follows.
\end{proof}

Every $(p,2,J,\Delta_k)$-atom is a
$(p,2,J,\Delta_k,\varepsilon)$-molecule for each
$\varepsilon>0$. Hence the preceding proposition
also gives
\begin{equation}\label{atom_in_Hp}
    \|\boldsymbol a\|_{\mathbb H^p_{\mathrm{Dunkl}}}
    \le C
\end{equation}
uniformly over all such atoms, in agreement with
\cite[Proposition 4.2]{DzHL_atom}.

In the multiplier proof we apply
Proposition~\ref{teo:operator_atomic_decomposition}
with atomic order $2J$, and
Proposition~\ref{prop:m_in_Hp} with molecular order $J$.
This allows the additional cancellation of the input
atom to provide the low-frequency decay while
retaining the molecular structure of its image.

\section{Weighted estimates for the dyadic kernels}\label{sec:weighted}
Recall the definitions of $\omega$ and $\Omega$ from
\eqref{eq:omega_Omega}, the joint kernel estimate
\eqref{eq:joint_kernel_input}, and the volume
comparison \eqref{eq:Omega_ball_comparison}. All Sobolev norms below are the classical norms
$W_2^u$ defined in the introduction. Let $\mathcal W=\bigcup_{\alpha\in R}\alpha^\perp$. 
For a compactly supported function $a\in L^2(d\xi)$,
define
\begin{equation}\label{eq:annular_kernel_definition}
    \mathscr K[a](\mathbf x,\mathbf y)
    =c_k^{-2}\int_{\mathbb R^N}
       a(\xi)E(i\xi,\mathbf x)E(-i\xi,\mathbf y)\,dw(\xi).
\end{equation}
The integral is absolutely convergent, since $\omega$
is bounded on compact sets and
$|E(i\xi,\mathbf x)|\le1$ for real $\xi,\mathbf x$.
When $a$ is bounded, $\mathscr K[a]$ is the integral
kernel of $\mathcal T_a=\mathcal F^{-1}(a\mathcal F)$:
\begin{equation}\label{eq:multiplier_kernel_representation}
    \mathcal T_af(\mathbf x)
    =\int_{\mathbb R^N}
       \mathscr K[a](\mathbf x,\mathbf y)f(\mathbf y)
       \,dw(\mathbf y),
    \qquad f\in L^1(dw)\cap L^2(dw).
\end{equation}
In terms of the translation defined in
\begin{equation}\label{eq:multiplier_kernel_translation}
    \mathscr K[a](\mathbf x,\mathbf y)
    =c_k^{-1}
       \tau_{\mathbf x}(\mathcal F^{-1}a)(-\mathbf y).
\end{equation}

\begin{lemma}[Proposition 4.3, ~\cite{CLWW}]\label{lem:weighted_annular_kernel}
Let $\mathcal A$ be a fixed compact annulus centred at
the origin in $\mathbb R^N\setminus\{0\}$., and let $u\ge0$.
There exists a constant $C>0$ such that, for every
$a\in W_2^u(\mathbb R^N)$ supported in $\mathcal A$,
\begin{equation}\label{eq:weighted_annular_kernel}
    \bigl\|(1+d(\cdot,\mathbf y))^u
       \mathscr K[a](\cdot,\mathbf y)\bigr\|_{L^2(dw)}
    \le C\Omega(\mathbf y)^{-1/2}\|a\|_{W_2^u},
    \qquad \mathbf y\in\mathbb R^N.
\end{equation}
The constant depends only on $R$, $k$, $u$, and
$\mathcal A$.
\end{lemma}

\begin{proof}
This follows from \cite[Proposition~4.3]{CLWW}, applied
with $g=\mathrm{id}$, together with a standard density
argument using smooth approximations supported in a fixed
slightly larger annulus.
\end{proof}

Choose a radial function
$\phi\in C_c^\infty(\mathbb R^N\setminus\{0\})$,
supported in $\{\xi:1/2\le\|\xi\|\le2\}$, such that
\[
    \sum_{\ell\in\mathbb Z}\phi(2^{-\ell}\xi)=1,
    \qquad \xi\ne0.
\]
Set
\[
    m_\ell(\xi)=m(\xi)\phi(2^{-\ell}\xi),
    \qquad
    \lambda_\ell=2^\ell,
\]
and
\[
    a_\ell(\eta)
    =m_\ell(\lambda_\ell\eta)
    =m(\lambda_\ell\eta)\phi(\eta).
\]
Suppose that the quantity $M$ in
\eqref{eq:assumption} is finite. By
Remark~\ref{rem:any_eta},
\begin{equation}\label{eq:normalized_symbol_bound}
    \sup_{\ell\in\mathbb Z}\|a_\ell\|_{W_2^s}
    \le C_{\psi,\phi,s}M.
\end{equation}

The following corollary records the rescaled forms of
Lemma~\ref{lem:weighted_annular_kernel} that will be used
in the proof of the multiplier theorem. Its purpose is to
make explicit the dependence on the dyadic scale
$\lambda_\ell=2^\ell$, both for the localized multiplier
$m_\ell$ and for the modified multiplier
$m_\ell(\xi)\|\xi\|^{2J}$.

\begin{corollary}\label{cor:weighted_dyadic_kernel}
Let $0\le\sigma\le s$, and let
$\mathscr K_\ell=\mathscr K[m_\ell]$. Then
\begin{equation}\label{eq:weighted_dyadic_kernel}
    \bigl\|(1+\lambda_\ell d(\cdot,\mathbf y))^\sigma
       \mathscr K_\ell(\cdot,\mathbf y)\bigr\|_{L^2(dw)}
    \le CM\,w(B(\mathbf y,\lambda_\ell^{-1}))^{-1/2}.
\end{equation}
Moreover, for every integer $J\ge0$, if
\[
    n_\ell(\xi)=m_\ell(\xi)\|\xi\|^{2J},
    \qquad
    \mathscr L_\ell=\mathscr K[n_\ell],
\]
then
\begin{equation}\label{eq:weighted_dyadic_power_kernel}
    \bigl\|(1+\lambda_\ell d(\cdot,\mathbf y))^\sigma
       \mathscr L_\ell(\cdot,\mathbf y)\bigr\|_{L^2(dw)}
    \le CM\lambda_\ell^{2J}
       w(B(\mathbf y,\lambda_\ell^{-1}))^{-1/2}.
\end{equation}
The constants are independent of $\ell$, $\mathbf y$,
and $m$; the second constant may also depend on $J$.
\end{corollary}

\begin{proof}
Recall that
\[
    a_\ell(\eta)
    =m_\ell(\lambda_\ell\eta)
    =m(\lambda_\ell\eta)\phi(\eta).
\]
Thus all the normalized symbols $a_\ell$ are supported
in the same fixed compact annulus centred at the origin,
\[
    \operatorname{supp}a_\ell
    \subseteq\operatorname{supp}\phi
    \subseteq
    \mathcal A
    :=\{\eta\in\mathbb R^N:1/2\le\|\eta\|\le2\}.
\]
Moreover, by \eqref{eq:normalized_symbol_bound} and
$\sigma\le s$,
\[
    \|a_\ell\|_{W_2^\sigma}
    \le\|a_\ell\|_{W_2^s}\le CM.
\]
Consequently, Lemma~\ref{lem:weighted_annular_kernel}
applies to every $a_\ell$ with the same constant:
\begin{equation}\label{eq:normalized_kernel_estimate}
    \bigl\|(1+d(\cdot,\mathbf y))^\sigma
       \mathscr K[a_\ell](\cdot,\mathbf y)\bigr\|_{L^2(dw)}
    \le CM\Omega(\mathbf y)^{-1/2},
    \qquad \mathbf y\in\mathbb R^N.
\end{equation}

The homogeneity of $E$ and $dw$ gives
\[
    \mathscr K_\ell(\mathbf x,\mathbf y)
    =\lambda_\ell^{\mathbf N}
       \mathscr K[a_\ell]
          (\lambda_\ell\mathbf x,\lambda_\ell\mathbf y).
\]
Since
\[
    d(\lambda_\ell\mathbf x,\lambda_\ell\mathbf y)
    =\lambda_\ell d(\mathbf x,\mathbf y),
\]
the change of variables
and \eqref{eq:normalized_kernel_estimate} yield
\begin{align*}
    &\bigl\|(1+\lambda_\ell d(\cdot,\mathbf y))^\sigma
       \mathscr K_\ell(\cdot,\mathbf y)\bigr\|_{L^2(dw)}=
       \lambda_\ell^{\mathbf N/2}
       \bigl\|(1+d(\cdot,\lambda_\ell\mathbf y))^\sigma
       \mathscr K[a_\ell](\cdot,\lambda_\ell\mathbf y)
       \bigr\|_{L^2(dw)}\le
       CM\lambda_\ell^{\mathbf N/2}
       \Omega(\lambda_\ell\mathbf y)^{-1/2}.
\end{align*}
By \eqref{eq:Omega_ball_comparison} and
\eqref{eq:t_ball},
\[
    \Omega(\lambda_\ell\mathbf y)
    \asymp w(B(\lambda_\ell\mathbf y,1))
    =\lambda_\ell^{\mathbf N}
       w(B(\mathbf y,\lambda_\ell^{-1})).
\]
This proves \eqref{eq:weighted_dyadic_kernel}.

For the second assertion, set
\[
    b_\ell(\eta)
    =n_\ell(\lambda_\ell\eta)
    =\lambda_\ell^{2J}\|\eta\|^{2J}a_\ell(\eta).
\]
These symbols are also supported in the same fixed
annulus $\mathcal A$. Choose
$\chi_0\in C_c^\infty(\mathbb R^N\setminus\{0\})$
equal to one on a neighbourhood of $\mathcal A$.
Then
\[
    b_\ell(\eta)
    =\lambda_\ell^{2J}
       \bigl(\chi_0(\eta)\|\eta\|^{2J}\bigr)a_\ell(\eta).
\]
Multiplication by the fixed smooth compactly supported
function $\chi_0(\eta)\|\eta\|^{2J}$ is bounded on
$W_2^\sigma$. Therefore,
\[
    \|b_\ell\|_{W_2^\sigma}
    \le C_J\lambda_\ell^{2J}\|a_\ell\|_{W_2^\sigma}
    \le C_JM\lambda_\ell^{2J}.
\]
Applying Lemma~\ref{lem:weighted_annular_kernel}
to $b_\ell$ gives
\[
    \bigl\|(1+d(\cdot,\mathbf y))^\sigma
       \mathscr K[b_\ell](\cdot,\mathbf y)\bigr\|_{L^2(dw)}
    \le C_JM\lambda_\ell^{2J}\Omega(\mathbf y)^{-1/2},
\]
uniformly in $\ell$ and $\mathbf y$. Finally,
\[
    \mathscr L_\ell(\mathbf x,\mathbf y)
    =\lambda_\ell^{\mathbf N}
       \mathscr K[b_\ell]
          (\lambda_\ell\mathbf x,\lambda_\ell\mathbf y).
\]
The same change of variables and ball measure comparison
as above prove \eqref{eq:weighted_dyadic_power_kernel}.
\end{proof}

\section{Proof of H\"ormander's multiplier theorem}\label{sec:Hormander}

\begin{proof}[Proof of Theorem~\ref{teo:teo_main}]
Let $0<p\le1$, and assume that
\[
    M=\sup_{t>0}
       \|\psi(\cdot)m(t\cdot)\|_{W_2^s}<\infty
\]
for some
\[
    s>\mathbf N\left(\frac1p-\frac12\right)=: A_p.
\]

Choose $\sigma$ and a positive integer $J$ such that
\begin{equation}\label{eq:new_parameter_choice}
    A_p<\sigma<s,
    \qquad 2J>\sigma.
\end{equation}
Then
\[
    \sigma>\frac{\mathbf N}{2},
    \qquad
    J>\frac{\mathbf N(2-p)}{4p}.
\]
Set
\begin{equation}\label{eq:new_molecular_epsilon}
    \varepsilon=\sigma-A_p>0.
\end{equation}
All constants below may depend on the fixed parameters
$R,k,p,s,\sigma,J,\psi,\phi$, but not on $m$, the atom,
its associated ball, or the dyadic indices.

We first note that
\begin{equation}\label{eq:new_multiplier_Linfty}
    \|m\|_{L^\infty}\le CM.
\end{equation}
Indeed, by Remark~\ref{rem:any_eta},
\[
    \sup_{\ell\in\mathbb Z}
       \|m(2^\ell\cdot)\phi\|_{W_2^s}\le CM.
\]
Since $s>\mathbf N/2\ge N/2$, the classical Sobolev
embedding theorem gives uniform $L^\infty$ bounds
for these localized symbols. The identity
\[
    m(\xi)=\sum_{\ell\in\mathbb Z}
       \phi(2^{-\ell}\xi)m(\xi),
    \qquad \xi\ne0,
\]
and the finite overlap of the annuli imply
\eqref{eq:new_multiplier_Linfty}. Consequently,
Plancherel's theorem (see~\eqref{eq:Plancherel}) yields
\begin{equation}\label{eq:new_multiplier_L2}
    \|\mathcal T_m\|_{L^2(dw)\to L^2(dw)}\le CM.
\end{equation}
Let $\boldsymbol a=\Delta_k^{2J}\boldsymbol b$ be a
$(p,2,2J,\Delta_k)$-atom associated with
$B=B(\mathbf x_0,r)$. Thus
$\boldsymbol b\in\mathcal D(\Delta_k^{2J})$,
\begin{equation}\label{eq:new_atom_support}
    \operatorname{supp}\boldsymbol b\subset\mathcal O(B),
\end{equation}
and
\begin{equation}\label{eq:new_atom_size}
    \|(r^2\Delta_k)^q\boldsymbol b\|_{L^2(dw)}
    \le r^{4J}w(B)^{1/2-1/p},
    \qquad q=0,\ldots,2J.
\end{equation}
Moreover,
\begin{equation}\label{eq:new_atom_power_support}
    \operatorname{supp}(\Delta_k^q\boldsymbol b)
       \subset\mathcal O(B),
    \qquad q=0,\ldots,2J.
\end{equation}
To justify this for the domain functions under
consideration, observe that $\mathcal O(B)$ is
$G$-invariant. Dunkl differentiation preserves support in a closed $G$-invariant set:
ordinary differentiation preserves support, and the
reflection terms only reflect that support by elements
of $G$. Applying this observation repeatedly proves
\eqref{eq:new_atom_power_support}.

We shall prove that $ \mathcal T_m\boldsymbol a$ is proportional to $(p,2,J,\varepsilon)$-molecule.

Define
$
    \boldsymbol b'=\mathcal T_m\Delta_k^J\boldsymbol b
$.
The characterization of
$\mathcal D(\Delta_k^J)$, see \eqref{eq:domain}, and
\eqref{eq:new_multiplier_Linfty} imply
$
    \boldsymbol b'\in\mathcal D(\Delta_k^J)
$.
Indeed, $\boldsymbol b'\in L^2(dw)$ and
\[
    \|\xi\|^{2J}\mathcal F\boldsymbol b'(\xi)
    =(-1)^Jm(\xi)\|\xi\|^{4J}\mathcal F\boldsymbol b(\xi)
       \in L^2(dw(\xi)).
\]
Commutation of the multipliers gives
\begin{equation}\label{eq:new_molecular_representation}
    \mathcal T_m\boldsymbol a=\Delta_k^J\boldsymbol b'.
\end{equation}

Now we verify the size condition in
Definition~\ref{def:molecule}, namely
\begin{equation}\label{eq:new_molecule_target}
    \|(r^2\Delta_k)^n\boldsymbol b'\|_{L^2(U_j(B))}
    \le CM r^{2J}2^{-j\varepsilon}
       w(2^jB)^{1/2-1/p},
\end{equation}
for $n=0,\ldots,J$ and $j\ge0$. Here
\[
    U_0(B)=\mathcal O(B),
    \qquad
    U_j(B)=\mathcal O(2^jB)\setminus
              \mathcal O(2^{j-1}B),\quad j\ge1.
\]

For $j=0,1$, equations \eqref{eq:new_multiplier_L2}
and \eqref{eq:new_atom_size} give
\begin{align*}
    \|(r^2\Delta_k)^n\boldsymbol b'\|_{L^2(U_j(B))}
    &\le
       CM r^{-2J}
       \|(r^2\Delta_k)^{J+n}\boldsymbol b\|_{L^2(dw)}\le
       CM r^{2J}w(B)^{1/2-1/p}.
\end{align*}
Since $w(2B)\asymp w(B)$, this proves
\eqref{eq:new_molecule_target} for these values of $j$.

Fix $j\ge2$. If $\mathbf x\in U_j(B)$ and
$\mathbf y\in\mathcal O(B)$, then
\begin{equation}\label{eq:new_orbit_separation}
\begin{split}
    d(\mathbf x,\mathbf y)
    &\ge d(\mathbf x,\mathbf x_0)
          -d(\mathbf y,\mathbf x_0)\ge (2^{j-1}-1)r
     \ge 2^{j-2}r.
\end{split}
\end{equation}
We also need the comparison
\begin{equation}\label{eq:new_w(B)olume_comparison}
    w(B(\mathbf y,\lambda^{-1}))^{-1/2}
    \le Cw(B)^{-1/2}
       \max\{1,(\lambda r)^{\mathbf N/2}\},
    \qquad \mathbf y\in\mathcal O(B),\quad\lambda>0.
\end{equation}
To prove it, choose $g\in G$ such that
$\mathbf y\in B(g(\mathbf x_0),r)$. Then
\[
    B(g(\mathbf x_0),r)\subset B(\mathbf y,2r),
    \qquad w(B(g(\mathbf x_0),r))=w(B).
\]
If $\lambda r\ge1$, the growth estimate
\eqref{eq:growth} implies
\[
    w(B)\le w(B(\mathbf y,2r))
       \le C(\lambda r)^{\mathbf N}
          w(B(\mathbf y,\lambda^{-1})).
\]
If $\lambda r<1$, inclusion and doubling give
\[
    w(B)\le w(B(\mathbf y,2r))
       \le w(B(\mathbf y,2\lambda^{-1}))
       \le Cw(B(\mathbf y,\lambda^{-1})).
\]
This proves \eqref{eq:new_w(B)olume_comparison}.

Fix $n\in\{0,\ldots,J\}$ and put
\[
    f_n=(r^2\Delta_k)^n\Delta_k^J\boldsymbol b,
    \qquad
    g_n=(r^2\Delta_k)^n\boldsymbol b.
\]
By \eqref{eq:new_atom_power_support}, both functions
are supported in $\mathcal O(B)$. Moreover,
\eqref{eq:new_atom_size} implies
\[
    \|f_n\|_{L^2(dw)}\le r^{2J}w(B)^{1/2-1/p},
    \qquad
    \|g_n\|_{L^2(dw)}\le r^{4J}w(B)^{1/2-1/p}.
\]
Since $w(\mathcal O(B))\le |G|w(B)$, the Cauchy--Schwarz
inequality gives
\begin{equation}\label{eq:new_atom_L1_bounds}
    \|f_n\|_1\le Cr^{2J}w(B)^{1-1/p},
    \qquad
    \|g_n\|_1\le Cr^{4J}w(B)^{1-1/p}.
\end{equation}

Recall the symbols
\[
    m_\ell(\xi)=m(\xi)\phi(2^{-\ell}\xi),
    \qquad
    n_\ell(\xi)=m_\ell(\xi)\|\xi\|^{2J},
    \qquad
    \lambda_\ell=2^\ell,
\]
and their integral kernels $\mathscr K_\ell$ and
$\mathscr L_\ell$. Corollary~\ref{cor:weighted_dyadic_kernel}
applies with the parameters $\sigma$ and $J$ chosen
above.

The dyadic partition of unity gives
\begin{equation}\label{eq:new_dyadic_atom_decomposition}
    (r^2\Delta_k)^n\boldsymbol b'
    =\mathcal T_mf_n
    =\sum_{\ell\in\mathbb Z}\mathcal T_{m_\ell}f_n
    \quad\text{in }L^2(dw).
\end{equation}
Indeed, the finite-overlap property gives a uniform
bound for the partial sums of
$\sum_\ell\phi(2^{-\ell}\xi)$, and these sums converge
to one for $\xi\ne0$. Dominated convergence on the
transform side and Plancherel's theorem prove the
asserted $L^2$ convergence.

\medskip
\noindent
\textit{Case 1: $\lambda_\ell r\ge1$.}
Write $u=\lambda_\ell r$. Since $f_n\in L^1(dw)\cap
L^2(dw)$, the integral-kernel representation~\eqref{eq:multiplier_kernel_representation},
Minkowski's inequality, \eqref{eq:new_orbit_separation},
and \eqref{eq:weighted_dyadic_kernel} yield
\begin{align*}
    \|\mathcal T_{m_\ell}f_n\|_{L^2(U_j(B))}
    &\le
       \int_{\mathcal O(B)}
       |f_n(\mathbf y)|
       \|\mathscr K_\ell(\cdot,\mathbf y)\|_{L^2(U_j(B))}
       \,dw(\mathbf y)\\
    &\le
       CM(1+2^ju)^{-\sigma}
       \int_{\mathcal O(B)}
       |f_n(\mathbf y)|
       w(B(\mathbf y,\lambda_\ell^{-1}))^{-1/2}
       \,dw(\mathbf y).
\end{align*}
The constants absorb the fixed factor in
\eqref{eq:new_orbit_separation}. Using
\eqref{eq:new_w(B)olume_comparison} and
\eqref{eq:new_atom_L1_bounds}, we obtain
\begin{equation}\label{eq:new_high_frequency_term}
    \|\mathcal T_{m_\ell}f_n\|_{L^2(U_j(B))}
    \le CM r^{2J}w(B)^{1/2-1/p}
       u^{\mathbf N/2}(1+2^ju)^{-\sigma}.
\end{equation}
Since $\sigma>\mathbf N/2$,
\begin{equation}\label{eq:new_high_frequency_sum}
\begin{split}
    \sum_{\lambda_\ell r\ge1}
       \|\mathcal T_{m_\ell}f_n\|_{L^2(U_j(B))}
    &\le
       CM r^{2J}w(B)^{1/2-1/p}2^{-j\sigma}
       \sum_{\lambda_\ell r\ge1}
          (\lambda_\ell r)^{\mathbf N/2-\sigma}\\
    &\le
       CM r^{2J}w(B)^{1/2-1/p}2^{-j\sigma}.
\end{split}
\end{equation}
The geometric-series bound is uniform in $r$:
if $\ell_0$ is the smallest integer with
$2^{\ell_0}r\ge1$, then $1\le2^{\ell_0}r<2$.

\medskip
\noindent
\textit{Case 2: $\lambda_\ell r<1$.}
Note that
\begin{equation}\label{eq:new_low_frequency_identity}
    \mathcal T_{m_\ell}f_n
    =(-1)^J\mathcal T_{n_\ell}g_n.
\end{equation}
Applying Minkowski's inequality and
\eqref{eq:weighted_dyadic_power_kernel}, together
with \eqref{eq:new_orbit_separation}, gives
\begin{align*}
    \|\mathcal T_{m_\ell}f_n\|_{L^2(U_j(B))}
    &\le
       CM\lambda_\ell^{2J}(1+2^ju)^{-\sigma}
       \int_{\mathcal O(B)}
       |g_n(\mathbf y)|
       w(B(\mathbf y,\lambda_\ell^{-1}))^{-1/2}
       \,dw(\mathbf y).
\end{align*}
Here $u<1$. Therefore
\eqref{eq:new_w(B)olume_comparison} and
\eqref{eq:new_atom_L1_bounds} imply
\begin{equation}\label{eq:new_low_frequency_term}
\begin{split}
    \|\mathcal T_{m_\ell}f_n\|_{L^2(U_j(B))}
    &\le
       CM\lambda_\ell^{2J}
       (1+2^ju)^{-\sigma}r^{4J}w(B)^{1/2-1/p}\\
    &=CM r^{2J}w(B)^{1/2-1/p}
       u^{2J}(1+2^ju)^{-\sigma}.
\end{split}
\end{equation}
Using $(1+2^ju)^{-\sigma}\le2^{-j\sigma}u^{-\sigma}$
and $2J>\sigma$, we obtain
\begin{equation}\label{eq:new_low_frequency_sum}
\begin{split}
    \sum_{\lambda_\ell r<1}
       \|\mathcal T_{m_\ell}f_n\|_{L^2(U_j(B))}
    &\le
       CM r^{2J}w(B)^{1/2-1/p}2^{-j\sigma}
       \sum_{\lambda_\ell r<1}
          (\lambda_\ell r)^{2J-\sigma}\\
    &\le
       CM r^{2J}w(B)^{1/2-1/p}2^{-j\sigma}.
\end{split}
\end{equation}
Again the geometric-series bound is independent of $r$.

Equations \eqref{eq:new_high_frequency_sum} and
\eqref{eq:new_low_frequency_sum} show absolute
convergence of the dyadic series in $L^2(U_j(B))$.
Its sum agrees with the restriction of the $L^2(dw)$
limit in \eqref{eq:new_dyadic_atom_decomposition}.
Consequently,
\begin{equation}\label{eq:new_annular_atom_bound}
    \|(r^2\Delta_k)^n\boldsymbol b'\|_{L^2(U_j(B))}
    \le CM r^{2J}w(B)^{1/2-1/p}2^{-j\sigma}.
\end{equation}

By the growth estimate \eqref{eq:growth},
\begin{align*}
    w(B)^{1/2-1/p}
    &=
       \left(\frac{w(2^jB)}{w(B)}\right)^{1/p-1/2}
       w(2^jB)^{1/2-1/p}\le C2^{jA_p}w(2^jB)^{1/2-1/p}.
\end{align*}
Substitution into \eqref{eq:new_annular_atom_bound}
proves
\[
    \|(r^2\Delta_k)^n\boldsymbol b'\|_{L^2(U_j(B))}
    \le CM r^{2J}2^{-j(\sigma-A_p)}
       w(2^jB)^{1/2-1/p}.
\]
Together with the estimates for $j=0,1$, this is
\eqref{eq:new_molecule_target}.

Thus \eqref{eq:new_molecular_representation} and
Definition~\ref{def:molecule} show that
$(CM)^{-1}\mathcal T_m\boldsymbol a$ is a
$(p,2,J,\Delta_k,\varepsilon)$-molecule associated
with $B$. Since
\[
    J>\frac{\mathbf N(2-p)}{4p}
    \quad\text{and}\quad \varepsilon>0,
\]
Proposition~\ref{prop:m_in_Hp} applies and gives
\begin{equation}\label{eq:new_multiplier_on_atoms}
    \|\mathcal T_m\boldsymbol a\|_{\mathbb H^p_{\mathrm{Dunkl}}}
    \le CM,
\end{equation}
uniformly over all $(p,2,2J,\Delta_k)$-atoms.

It remains to pass from atoms to arbitrary functions.
Let $f\in\mathbb H^p_{\mathrm{Dunkl}}$.
By Theorem~\ref{teo:operator_atomic_decomposition},
applied with atomic order $2J$, there is a decomposition
\[
    f=\sum_{\nu=1}^\infty c_\nu\boldsymbol a_\nu
    \quad\text{in }L^2(dw),
    \qquad
    \sum_{\nu=1}^\infty|c_\nu|^p
    \le C\|f\|_{\mathbb H^p_{\mathrm{Dunkl}}}^p.
\]
Put
\[
    f_L=\sum_{\nu=1}^L c_\nu\boldsymbol a_\nu.
\]
For the square function $S$ defined in
\eqref{eq:square_conic}, Minkowski's inequality in
the defining integral gives subadditivity. Hence,
using $0<p\le1$ and
\eqref{eq:new_multiplier_on_atoms}, we obtain
\begin{equation}\label{eq:new_finite_sum_bound}
\begin{split}
    \|S(\mathcal T_mf_L)\|_{L^p(dw)}^p
    &\le
       \sum_{\nu=1}^L |c_\nu|^p
       \|S(\mathcal T_m\boldsymbol a_\nu)\|_{L^p(dw)}^p\le CM^p\sum_{\nu=1}^L|c_\nu|^p.
\end{split}
\end{equation}

By \eqref{eq:new_multiplier_L2},
$\mathcal T_mf_L\to\mathcal T_mf$ in $L^2(dw)$.
Recall also that
\[
    |Sh-Sg|\le S(h-g),
    \qquad
    \|Sv\|_{L^2(dw)}\le C\|v\|_{L^2(dw)}.
\]
The latter estimate follows from Fubini's theorem,
doubling, and Plancherel's theorem applied to
$t^2\Delta_k H_{t^2}v$. Therefore
\[
    S(\mathcal T_mf_L)\longrightarrow
    S(\mathcal T_mf)
    \quad\text{in }L^2(dw).
\]
Passing to an almost everywhere convergent subsequence
and applying Fatou's lemma in
\eqref{eq:new_finite_sum_bound}, we conclude that
\[
    \|\mathcal T_mf\|_{\mathbb H^p_{\mathrm{Dunkl}}}^p
    \le CM^p\sum_{\nu=1}^\infty|c_\nu|^p
    \le CM^p\|f\|_{\mathbb H^p_{\mathrm{Dunkl}}}^p.
\]
Thus $\mathcal T_m$ is bounded on
$\mathbb H^p_{\mathrm{Dunkl}}$. Since
$H^p_{\mathrm{Dunkl}}$ is its completion,
$\mathcal T_m$ has a unique bounded extension to
$H^p_{\mathrm{Dunkl}}$, satisfying
\[
    \|\mathcal T_m\|_{H^p_{\mathrm{Dunkl}}
                         \to H^p_{\mathrm{Dunkl}}}
    \le CM.
\]
\end{proof}

\section*{Declarations}

\subsection*{Use of Generative AI and AI-Assisted Technologies}
During the preparation of this manuscript, the authors utilized Google's Gemini (Thinking 3.6) and OpenAI's GPT (GPT-6-Astra) to assist with language refinement, typesetting equations, and drafting preliminary standard technical lemmas. 

The authors reviewed and edited all content as necessary and takes full responsibility for the accuracy, mathematical correctness, and overall integrity of the final published work.

\subsection*{Competing Interests}
The authors declare no competing financial or non-financial interests directly relevant to the content of this article.

\subsection*{Data Availability}
Data sharing is not applicable to this article as no datasets were generated or analyzed during the current study.

\end{document}